\documentclass[12pt]{article}   	
\usepackage{geometry}                		
\usepackage[parfill]{parskip}    		
\usepackage{graphicx}				
\usepackage[hyperindex=true,pagebackref,colorlinks=true,pdfpagemode=none,urlcolor=blue,linkcolor=blue,citecolor=blue,pdfstartview=FitH]{hyperref}
\usepackage{amsmath,amsfonts}
\usepackage{color}
\usepackage{amsthm}
\usepackage{boolexpr}
\usepackage{amssymb,latexsym,pstricks,pst-plot}
\usepackage[english]{babel}
\usepackage{pgf}
\usepackage{tikz}
\usetikzlibrary{arrows,automata}
\usepackage[latin1]{inputenc}
\usepackage[UKenglish]{isodate}
\allowdisplaybreaks

\DeclareMathOperator{\ord}{ord}

\usepackage{amssymb}
\newtheorem{theorem}{Theorem}[section]
\newtheorem{lemma}[theorem]{Lemma}

\newtheorem{corollary}[theorem]{Corollary}

\newcommand{\Legendre}[2]{\genfrac{(}{)}{}{}{#1}{#2}}

\title{The Sum of Squares function in the ring $\mathbb{Z}_{n}$}
\author{Rob Burns}

\begin{document}
\maketitle
\begin{abstract}
We consider the sum of squares function in the ring $\mathbb{Z}_{n}$. We determine formulae in a number of cases when $n$ is a power of a prime.
\end{abstract}

\section{Introduction}
\label{intro}
The representation of elements of a ring as a sum of squares (or more generally as a sum of powers) is a very large and old subject. In the following we attempt to give a flavour of the type of questions that have been studied. It is not intended as a complete survey. 

Restricting ourselves to the integers for the moment, Fermat proved that every integer is either a square or a sum of $2$, $3$, or $4$ squares. Fermat also discovered that an odd prime $p$ can be expressed as the sum of $2$ squares if and only if $p \equiv 1 \pmod{4}$. The first proof of this result was provided by Euler using Fermat's method of infinite descent (\cite{Euler1752} \cite{Euler1754} \cite{Hardy:1979tk}). Euler extended this result to show that a positive integer can be represented as the sum of two squares if and only if each of its prime factors of the form $4k+3$ occurs as an even power (\cite{SUZUKI2007363}).

Diophantus, Bachet, Fermat, Descartes and Lagrange amongst others worked on the representation of integers as the sum of three or more squares. Legendre proved in his \textit{Essai sur la theorie des nombres }(1798) that a positive integer can be written as a sum of three squares if and only if it is not of the form $4^x(8y + 7)$ for nonnegative integers $x$ and $y$.

In 1770, prior to Legendre's result, Lagrange had proved that every positive integer can be written as the sum of at most four squares. This is known as Lagrange's four-square theorem or Bachet's conjecture (\cite{Ireland_1990}).

The asymptotic density of the positive integers which are the sum of two squares is zero. Landau showed that the asymptotic density of the positive integers which are the sum of three squares is $\frac{5}{6}$ (\cite{Landau1908}).

Moving away from the integers, we find that the representation of elements as a sum of squares has been studied in many other rings. 

In Hilbert's book \textit{Foundations of Geometry} (1899), he stated without proof that a totally positive element of any number field can be expressed as a sum of four squares in the field. Landau proved this for quadratic number fields in 1919 and Siegel extended the proof to all number fields in 1921 (\cite{Siegel_1921}). The situation is less straightforward when considering elements from the ring of integers rather than general elements in the number field. For example, in the field $\mathbb{Q}(i)$, which has $\mathbb{Z}(i)$ as its ring of integers, the element $i$ cannot be expressed as a sum of squares of elements from $\mathbb{Z}(i)$.

In 1940 Niven \cite{Niven1940} investigated imaginary quadratic number fields $K = \mathbb{Q}(\sqrt{-m})$ where $m$ is a square-free positive integer. If $\mathbb{Z}_K$ is the ring of integers of $K$, Niven showed that every element in $\mathbb{Z}_K$ can be expressed as a sum of three squares of elements from $\mathbb{Z}_K$ if m$ \equiv 3 \pmod 4$. He also showed that when $m \equiv 1 \pmod 4$, an element $x + y \sqrt{-m}$ in $\mathbb{Z}_K$ is a sum of three squares if and only if $y$ is an even integer.

Results for specific real quadratic number fields are known. For example, Fritz Gotzky showed in \cite{Gotzky1928} that every totally positive integer in  $K = \mathbb{Q}(\sqrt{5})$ can be represented as a sum of four squares of integers in $K$. Harvey Cohn considered $K = \mathbb{Q}(\sqrt{2})$ in \cite{Cohn1960} and proved that a totally positive integer $a = x + y \sqrt{2}$ can be written as a sum of four squares of integers in $K$ if and only if $y$ is even. He in fact provided a formula for the number of ways such an element $a$ could be written as a sum of four integral squares in $K$. He also established a partial result for $\mathbb{Q}(\sqrt{3})$.

In \cite{10.2307/43972685}, Cogdell proved that, in a totally real number field, all sufficiently large totally positive square free integers  that are sums of three squares locally everywhere are in fact sums of three squares globally. Schulze-Pillot showed in his survey article \cite{Schulze-Pillot2004} that the local-global principle does not apply to all totally positive integers. In particular, if $F = \mathbb{Q}(\sqrt{35})$ and $p$ is a prime satisfying $\Legendre{p}{7} = 1$, then no number of the form $7p^2$  is a sum of three integral squares in F, even though $7$ is a sum of three integral squares locally everywhere. 

In \cite{Colliot2009}, Colliot-Th\'el\'ene and Xu consider the connection between the representation of integral elements as a sum of squares and the Brauer-Manin obstruction. An appendix to the paper provides an example of the local-global principle in the setting of cyclotomic fields. The authors prove that an element $x$ in the ring of integers $O$ of a cyclotomic field is a sum of three squares of integers in $O$ if and only if $x$ is a sum of three squares in all local completions $O_v$ of $O$.

Hilbert's seventeenth problem asks whether every real positive definite polynomial, i.e., one which takes positive values only, can be expressed as a sum of squares of polynomials or, if not, whether it can be expressed as a sum of squares of rational functions. There are examples of real positive definite polynomials, such as the Motzkin polynomial $(x^2 + y^2 - 3z^2)x^2y^2 + z^6$, which cannot be expressed as a sum of squares of polynomials. In 1927, Artin \cite{Artin_1927} showed that such polynomials can always be expressed as a sum of squares of rational functions. Leep and Starr \cite{10.2307/2668853} provided examples of positive semidefinite polynomials in two variables which are sums of three rational squares, but not sums of polynomial squares.

The paper by Choi, Lam, Reznick and Rosenberg \cite{CHOI1980234} discusses the connection between representing an element $a$ in an integral domain $A$ as a sum of $n$ squares from the fraction field of $A$ and representing $a$ as a sum of $m \geq n$ squares in $A$ itself. The authors  prove that the former condition implies the latter for certain unique factorisation domains, regular semilocal domains and positive semidefinite polynomials in a polynomial ring over the reals.

The above questions can also be asked about matrix elements over a ring. In a 1968 paper, Carlitz \cite{Carlitz1968} proved that every two-by-two matrix over $\mathbb{Z}$ is a sum of a three squares. Griffin and Krusemeyer \cite{doi:10.1080/03081087708817172} give some circumstances under which a matrix is a sum of two squares. Newman \cite{newman1985} treated the case of $n \times n$ matrices over $\mathbb{Z}_2$ and $\mathbb{Z}$ and found the minimum number of squares needed to represent such matrices. More recently Katre and Garge \cite{Katre_2012} examined $M_n(R)$, the $n \times n$ matrices over $R$, where $R$ is a commutative, associative ring with unity. They provided trace conditions on a matrix in $M_n(R)$ which ensure it is a sum of $k$-th powers of matrices.

Hilbert's seventeenth problem can be extended to matrices. Gondard and Ribenboim \cite{Gondard1974} and Procesi and Schacher \cite{10.2307/1970962} independently proved that matrices with positive semidefinite polynomial function entries can be expressed as a sum of squares of symmetric matrices with rational function entries.

In a general setting, Fernando, Ruiz and Scheiderer \cite{10.2307/3844904} showed that certain excellent rings contain positive semidefinite elements which are not sums of squares.

We finally turn to the ring $\mathbb{Z}_n$ of integers modulo $n$. Harrington, Jones and Lamarche \cite{Harrington:2014wx} determined the values of $n$ for which every element of $\mathbb{Z}_n$ can be written as a sum of two squares. For a prime $p$ and positive integer $k$, Burns \cite{Burns:2017aa} determined which elements of $\mathbb{Z}_{p^k}$ can be written as a sum of two squares in the ring and then again determined the values of $n$ for which all elements of $\mathbb{Z}_n$ are a sum of two squares.  In \cite{Arias:2018aa} Arias,  Borja and Rubio counted the number of integers in $\mathbb{Z}_n$ that are in the image of various polynomials such as $x^2 + y^2$ and $x^2 + y^2 + z^2$. 

In the above contexts we also need to mention Waring's Problem which asks for the least positive integer $g$ such that every element of a given ring is the sum of at most $g$ squares (or higher powers) of elements from the ring. Waring's problem has been studied on each of the rings mentioned above.

We now introduce the sum of squares function $r_m$ which counts the number of ways an element of a ring can be written as a sum of $m$ squares, allowing for zero as one or more of the squares. Specifically,

$$
r_m(n) = \lvert \{ (x_1, x_2, ... , x_m) : \sum_{i=1}^{m} x_i ^2 = n \} \rvert.
$$

Much is known about $r_m$ over the integers. Jacobi expressed $r_m(n)$  in terms of divisor functions when m = 2,4,6 and 8, For example, in \cite{Jacobi1829} he proved that
$$
r_2(n) = 4( d_1(n) - d_3(n) )
$$
where $ d_1(n)$ and $d_3(n)$ are the number of divisors of $n$ congruent to $1 \pmod 4$ and \mbox{$3 \pmod 4$}, respectively. Expressions for $r_m (n)$ are also known for other values of $m$, however some of these involve terms that are not explicitly given, as they appear only as coefficients of modular functions. Milne \cite{10.2307/40817} and Ono \cite{Ono2002} established formulas for $r_{4s^2}$ and $r_{4s^2 + 4s}$ for every $s \geq 1$. Asymptotic expressions are also known for $r_m (n)$. The function $r_2$ is connected to the Gauss circle problem through the summation function $R(N) = \sum_{n \leq N} r_2(n)$. The generating function for $r_m$ can be expressed in terms of the Jacobi theta function:
$$
\sum_{n = 0}^{\infty} r_m(n) x^n = \theta_3^{m}(x).
$$
 
This paper considers the sum of squares function over the ring $\mathbb{Z}_{n}$ of integers modulo $n \in  \mathbb{Z}$.  We define the function $N_m(t, n)$ to be the number of ways of expressing an element $t \in \mathbb{Z}_{n}$ as a sum of $m$ squares of elements from $\mathbb{Z}_{n}$.  Let  $\mathbb{Z}^*_{n}$ be the set of units in  $\mathbb{Z}_{n}$, i.e. the set of elements $t \in \mathbb{Z}^*_{n}$ such that $\gcd(t, n) = 1$. We then define the function $N^*_m(t, n)$ to be the number of ways of expressing an element of $t \in \mathbb{Z}_{n}$ as a sum of $m$ squares of elements from $\mathbb{Z}^*_{n}$. We only need consider the case $n = p^k$, where $p$ is prime and $k \geq 1$. The general result can then be pieced together using the Chinese remainder theorem.

Lebesgue \cite{Dickson} calculated $N_m(t, p)$ when $p$ is an odd prime. His results correspond to Theorems \ref{k=1:m odd:p=1mod4} -  \ref{k=1:m even:p=3mod4} below. In 2014, T\'{o}th \cite{Toth2014} gave formulae for $N_2(t, p)$, $N_2(t, p^2)$ and $N_m(t, 2^k)$ where $p$ is prime and $m, k \geq 1$. He also calculated $N_m(t, p^k)$ for $t \in \{0, 1 \}$, $m \in \{2, 3, 4 \}$, $p$ prime and $k \geq 1$. Some of his results overlap with  Theorems \ref{m=2:p=1mod4}  -  \ref{m=3:p=3mod4} below. Yang and Tang \cite{Yang_2015} gave a formula for $N_2(t, n)$ when $n$ is any integer and $t$ is any element of $\mathbb{Z}_{n}$. Their results correspond to Theorems \ref{m=2:p=2}, \ref{m=2:p=1mod4}  and \ref{m=2:p=3mod4} below. Calder\'{o}n, Grau, Oller-Marc\'{e}n, and T\'{o}th \cite{Calderon2015} calculated $N_m(t, p^k)$ when $p$ is prime, $\gcd(t, p) = 1$ and $m, k \geq 1$. Grau and Oller-Marc\'{e}n \cite{GRAU2019427} completed the picture by calculating $N_m(t, p^k)$ for $p$ prime and $m, k \geq 1$.

Some authors have investigated  $N^*_m(t, n)$. As far as we can tell, there is no straightforward way of deriving  one of $N^*_m(t, n)$  or  $N_m(t, n)$ from the other. Yang and Tang \cite{Yang_2015} gave a formula for $N^*_2(t, n)$ when $n$ is any integer and $t$ is any element of $\mathbb{Z}_{n}$. Mollahajiaghaei \cite{Mollahajiaghaei_2017} and Li and Ouyang \cite{LI201841} generalised this by establishing a formula for $N^*_m(t, n)$ for any $m$. The formulae involve constants which are defined recursively.

We start with results for the ring $\mathbb{Z}_{2^k}$.

\bigskip

\begin{theorem}
\label{Nstar:m:p=2}
Let $k, m \geq 1$. Then,   
$$
N^*_m(t, 2^k) = 
\begin{cases}
    1  & \text{if } \, \, k = 1 \, \, \text{ and } \, \,  t \equiv m \pmod 2 \\
    2^m  & \text{if } \, \,  k = 2 \, \, \text{ and } \, \, t \equiv m \pmod 4 \\
    2^{2m + (m-1)(k-3)} & \text{if } \, \, k \geq 3 \, \, \text{ and }  \, \, t  \equiv m \pmod 8 \\
    0 & \text{otherwise.}
\end{cases}
$$
\end{theorem}

\bigskip

\begin{theorem}
\label{m:p=2}
$N_{m}(t, 2) = 2^{m-1}$ for $t \in \{0, 1 \}$. If $t \in \mathbb{Z}_{4}$, then,
$$
N_{m}(t, 4) = 
\begin{cases}
    2^{2m-2} + (-1)^{t/2} 2^{\frac{3m}{2} - 1} \cos (\frac{m \pi}{4} ) & \text{if } \, \, t \equiv 0 \pmod 2 \\
   2^{2m-2} + (-1)^{(t-1)/2} 2^{\frac{3m}{2} - 1} \sin (\frac{m \pi}{4} ) & \text{if } \, \, t \equiv 1 \pmod 2.
\end{cases}
$$
\end{theorem}
\bigskip

\begin{theorem}
\label{m=2:p=2}
Let $t \in \mathbb{Z}_{2^k}$ with $k \geq 2$. Write $t = 2^{\alpha}\beta$ where $2 \nmid \beta$. Then in $\mathbb{Z}_{2^k}$,
$$
N_{2}(t, 2^k) = 
\begin{cases}
    2^{k}  & \text{if } \, \, t = 0 \\
    2^{k+1}   & \text{if} \, \, 0 \leq \alpha < k - 1 \text{ and } \beta \equiv 1 \pmod 4 \\
    0   & \text{if} \, \, 0 \leq \alpha < k - 1 \text{ and } \beta \equiv 3 \pmod 4 \\
    2^{k}   & \text{if} \, \, \alpha = k - 1.
\end{cases}
$$
\end{theorem}

\bigskip

\begin{theorem}
\label{m=3:p=2}
Let $t \in \mathbb{Z}_{2^k}$ with $k \geq 3$ and write $t = 2^{\alpha}\beta$ where $2 \nmid \beta$. Then in $\mathbb{Z}_{2^k}$,
$$
N_{3}(t, 2^k) = 
\begin{cases}
    2^{3k/2}  & \text{if } \, \, t = 0  \text{ and } k  \text{ is even } \\
    2^{(3k+1)/2}  & \text{if } \, \, t = 0  \text{ and } k  \text{ is odd } \\
    3 \times 2^{2k - \alpha/2 - 1}   & \text{if} \, \, k - \alpha > 2, \alpha \text{ is even and } \beta \equiv 1 \pmod 4 \\
    2^{2k - \alpha/2}   & \text{if} \, \, k - \alpha > 2, \alpha \text{ is even and } \beta \equiv 3 \pmod 8 \\
    0   & \text{if} \, \, k - \alpha > 2, \alpha \text{ is even and } \beta \equiv 7 \pmod 8 \\
    3 \times 2^{2k - (\alpha + 1)/2}   & \text{if} \, \, k - \alpha > 2 \text{ and }\alpha \text{ is odd } \\
    3 \times 2^{3k/2}   & \text{if} \, \, k - \alpha = 2, \alpha \text{ is even and } \beta \equiv 1 \pmod 4 \\
    2^{3k/2}   & \text{if} \, \, k - \alpha = 2, \alpha \text{ is even and } \beta \equiv 3 \pmod 4 \\
    3 \times 2^{(3k + 1)/2}   & \text{if} \, \, k - \alpha = 2 \text{ and } \alpha \text{ is odd } \\
    2^{(3k + 1)/2}   & \text{if} \, \, k - \alpha = 1 \text{ and } \alpha \text{ is even } \\
    3 \times 2^{3k/2}   & \text{if} \, \, k - \alpha = 1 \text{ and } \alpha \text{ is odd.}
\end{cases}
$$
\end{theorem}

\bigskip

Next we present results for the ring $\mathbb{Z}_p$, where $p$ is an odd prime. When $t \in \mathbb{Z}_{p}$ the value of $N_m(t, p)$ depends on \mbox{$p \pmod 4$, $m \pmod 2$ and $\Legendre{t}{p}$}. For clarity, we have divided the result into four parts.

\bigskip

\begin{theorem}
\label{k=1:m odd:p=1mod4}
If $p$ is an odd prime with $p = 1 \pmod 4$ and $m \geq 0$, then in $\mathbb{Z}_p$,
$$
N_{2m+1}(t, p) = 
\begin{cases}
    p^{2m}  & \text{if } \, \, t = 0 \\
    p^{2m} - p^m   & \text{if} \, \, t \text{ is a non-residue} \pmod p \\
    p^{2m} + p^m   & \text{if} \, \, t \text{ is a residue} \pmod p.
\end{cases}
$$
\end{theorem}

\bigskip

\begin{theorem}
\label{k=1:m even:p=1mod4}
If $p$ is an odd prime with $p = 1 \pmod 4$ and $m \geq 1$, then in $\mathbb{Z}_p$,
$$
N_{2m}(t, p) = 
\begin{cases}
    p^{2m-1} + p^m - p^{m-1}   & \text{if } \, \, t = 0 \\
    p^{2m-1} - p^{m-1}    & \text{if} \, \, t \neq 0.
\end{cases}
$$
\end{theorem}

\bigskip

\begin{theorem}
\label{k=1:m odd:p=3mod4}
If $p$ is an odd prime with $p = 3 \pmod 4$ and $m \geq 0$, then in $\mathbb{Z}_p$
$$
N_{2m+1}(t, p) = 
\begin{cases}
    p^{2m}   & \text{if } \, \, t = 0 \\
    p^{2m} + (-1)^{m+1} p^m   & \text{if} \, \, t \text{ is a non-residue} \pmod p \\
    p^{2m} + (-1)^{m} p^m   & \text{if} \, \, t \text{ is a residue} \pmod p.
\end{cases}
$$
\end{theorem}

\bigskip

\begin{theorem}
\label{k=1:m even:p=3mod4}
If $p$ is an odd prime with $p = 3 \pmod 4$ and $m \geq 1$, then in $\mathbb{Z}_p$
$$
N_{2m}(t, p) = 
\begin{cases}
    p^{2m-1} + (-1)^m p^m + (-1)^{m-1} p^{m-1}  & \text{if } \, \, t = 0 \\
    p^{2m-1} + (-1)^{m-1} p^{m-1}    & \text{if} \, \, t \neq 0.
\end{cases}
$$
\end{theorem}

\bigskip

If we restrict ourselves to the sum of two squares in the ring $\mathbb{Z}_{p^k}$ we have the following:

\bigskip

\begin{theorem}
\label{m=2:p=1mod4}
Let $p$ be an odd prime with $p = 1 \pmod 4$ and let $k \geq 1$. Let $t \in \mathbb{Z}_{p^k}$ and write
$t = p^{\alpha}\beta$ where $p \nmid \beta$. Then in $\mathbb{Z}_{p^k}$
$$
N_{2}(t, p^k) = 
\begin{cases}
    p^{k-1} \left( p(k+1) - k \right)  & \text{if } \, \, t = 0 \\
    (\alpha + 1)(p - 1)p^{k-1}    & \text{if} \, \, t \neq 0
\end{cases}
$$
\end{theorem}

\bigskip

\begin{theorem}
\label{m=2:p=3mod4}
Let $p$ be an odd prime with $p = 3 \pmod 4$ and let $k \geq 1$. Let $t \in \mathbb{Z}_{p^k}$ and write
$t = p^{\alpha}\beta$ where $p \nmid \beta$. Then in $\mathbb{Z}_{p^k}$
$$
N_{2}(t, p^k) = 
\begin{cases}
    p^{2 \lfloor \frac{k}{2} \rfloor}  & \text{if } \, \, t = 0 \\
    (p + 1)p^{k-1}    & \text{if} \, \, \alpha \, \, \text{is even} \\
    0    & \text{if} \, \, \alpha \, \, \text{is odd}
\end{cases}
$$
\end{theorem}

\bigskip

For the sum of three squares in the ring $\mathbb{Z}_{p^k}$ we have:

\bigskip

\begin{theorem}
\label{m=3:p=1mod4}
Let $p$ be an odd prime with $p = 1 \pmod 4$ and let $k \geq 1$. Let $t \in \mathbb{Z}_{p^k}$ and write
$t = p^{\alpha}\beta$ where $p \nmid \beta$. Then in $\mathbb{Z}_{p^k}$
$$
N_{3}(t, p^k) = 
\begin{cases}
    p^{2k} + p^{2k-1} - p^{\lceil\frac{3k}{2}\rceil - 1}  & \text{if } \, \, t = 0 \\
    (p^{2k - 1} - p^{2k - \frac{\alpha + 3}{2}})(p+1)   & \text{if} \, \, \alpha \text{ is odd} \\
    p^{2k-1}(p + 1)   & \text{if} \, \, \alpha \, \,\text{ is even and} \, \, \Legendre{\beta}{p} = 1 \\
    p^{2k-1}(p+1) - 2p^{2k - 1-\frac{\alpha}{2}}   & \text{if} \, \, \alpha \, \, \text{ is even and} \, \,\Legendre{\beta}{p} = -1.
\end{cases}
$$
\end{theorem}

\bigskip

\begin{theorem}
\label{m=3:p=3mod4}
Let $p$ be an odd prime with $p = 3 \pmod 4$ and let $k \geq 1$. Let $t \in \mathbb{Z}_{p^k}$ and write
$t = p^{\alpha}\beta$ where $p \nmid \beta$. Then in $\mathbb{Z}_{p^k}$,
$$
N_{3}(t, p^k) = 
\begin{cases}
    p^{2k} + p^{2k-1} - p^{\lceil\frac{3k}{2}\rceil - 1}  & \text{if } \, \, t = 0 \\
    (p^{2k - 1} - p^{2k - \frac{\alpha + 3}{2}})(p+1)   & \text{if} \, \, \alpha \, \, \text{ is odd} \\
    p^{2k-1}(p+1) - 2p^{2k - 1-\frac{\alpha}{2}}  & \text{if} \, \, \alpha \, \, \text{ is even and} \, \, \Legendre{\beta}{p} = 1 \\
    p^{2k-1}(p + 1) & \text{if} \, \, \alpha \, \, \text{ is even and} \, \, \Legendre{\beta}{p} = -1.
\end{cases}
$$
\end{theorem}

\bigskip

\section{Notation}
\label{prelim}

We will make use of the Legendre symbol $\Legendre{t}{p}$. The floor of the real number $x$ is written $\lfloor x \rfloor$ and defined as the largest integer $\leq x$. The ceiling of the real number $x$ is written $\lceil x \rceil$ and defined as the smallest integer $\geq x$. If $A$ is a set, we denote the number of elements in $A$ by $\lvert A \rvert$.

When discussing divisibility of an integer $x$ by a prime $p$, we will use the notation $p^{n} \| x$ to mean that $n$ is the highest power of $p$ dividing $x$, i.e. $p^n \mid x$ and $p^{n+1} \nmid x$. We will also use the notation $\ord_p(x) = n$ to mean that $n$ is the highest power of $p$ dividing $x$.

For integers $n \in  \mathbb{N}$, we define the sets $S_{m}(t, n)$ and $S^*_{m}(t, n)$ as follows:
\begin{equation}
\label{Sset}
S_{m}(t, n) := \, \,  \Big\{  (x_1, x_2, ... , x_m)  \in \mathbb{Z}_{n}: \sum_{i=1}^{m} x_i ^2 \equiv t \pmod {n} \Big\} 
\end{equation}

\begin{equation}
\label{Sstarset}
S^*_{m}(t, n) := \, \,  \Big\{  (x_1, x_2, ... , x_m)  \in \mathbb{Z}^*_{n}: \sum_{i=1}^{m} x_i ^2 \equiv t \pmod {n} \Big\} 
\end{equation}

We then have $N_{m}(t, n)  = \lvert S_{m}(t, n) \rvert$ and  $N^*_{m}(t, n)  = \lvert S^*_{m}(t, n) \rvert$.

\bigskip 

Our first observation is that the value of $N_{m}(t, p)$ depends only on $\Legendre{t}{p}$.

\bigskip

\begin{lemma}
\label{resornotmodp}
Let $m \geq 1$ and suppose $s, t \in \mathbb{Z}_{p}$ with $\Legendre{s}{p} = \Legendre{t}{p}$. Then $N_{m}(s, p) = N_{m}(t, p)$.
\end{lemma}
\begin{proof}
If $\Legendre{s}{p} = \Legendre{t}{p}$ then $t/s$ is a quadratic residue $\pmod p$. Let $z = \sqrt{t/s}$ in $\mathbb{Z}_{p}$. Then
$$
\sum_{i=1}^{m} (x_i)^2 = s \iff \sum_{i=1}^{m} (x_i * z)^2 = t.
$$
\end{proof}

\bigskip

According to the lemma above, all information in $\mathbb{Z}_{p}$  about the function $N_{m}$ is contained in a $3 \times 1$ vector which we will call $\gamma_{m}$ and define by
\begin{equation}
  \label{gamma}
\gamma_{m} := \begin{pmatrix} N_{m}(0, p) \\ N_{m}(s, p) \\ N_{m}(t, p) \end{pmatrix}
\end{equation}
where $s$ is any non-residue and $t$ is any residue $\pmod p$. We denote the components of $\gamma_m$ by $\gamma_{m, 1}$, $\gamma_{m, 2}$ and $\gamma_{m, 3}$.

\bigskip

\section{The $\mathbb{Z}_{2^k}$ case: proof of theorems \ref{Nstar:m:p=2} - \ref{m=3:p=2}}
\label{Z_2k}

\subsection{Preliminaries}
We first provide some results which are required for the proof of theorems  \ref{Nstar:m:p=2} - \ref{m=3:p=2}. The following classical result was published in 1834 \cite{ramus1834}.
\begin{equation}
  \tag{Ramus' Identity}
  \sum_j \binom{k}{t + nj} = \frac{1}{n} \sum_{j=1}^{n} (2 \cos (\frac{j \pi}{n}))^k \cos (\frac{(k - 2t) j \pi }{n})
  \label{eqn:Ramus}
\end{equation}

\bigskip

\begin{lemma}
\label{x2=1mod8}
Let $t \in \mathbb{Z}_{2^k}$ with  $k \geq 3$ and $2 \nmid t$. Then,  
$$
N_1(t, 2^k) = \Big\lvert \Big\{ x \in \mathbb{Z}_{2^k}: x^2 \equiv t \pmod {2^k} \Big\} \Big\rvert =
\begin{cases}
   4 & \text{if } \, \, t \equiv 1 \pmod 8  \\
   0 & \text{otherwise.}
\end{cases}
$$
\end{lemma}
\begin{proof}
It is clear that any element which is $ \, \not\equiv 1 \pmod 8$ cannot be a square in  $\mathbb{Z}_{2^k}$. We will use induction on $k$ to show that any element of  $\mathbb{Z}_{2^k}$ which is $\equiv 1 \pmod 8$ is a square in  $\mathbb{Z}_{2^k}$. This holds for $k = 3$. Suppose it holds for $k$ and let $t \in \mathbb{Z}_{2^{k+1}}$ with $t \equiv 1 \pmod 8$. By induction there is a $y$ such that $y^2 \equiv t \pmod {2^k}$.  So $t = y^2 + 2^k r$ for some $r$. If $r$ is even then $y^2 \equiv t \pmod {2^{k+1}}$. If $r$ is odd, then
\begin{align*}
(y + 2^{k-1})^2 = & \, \,y^2 + 2^k y + 2^{2k - 2} \\
 = & \, \, t + 2^k(y - r  + 2^{k - 2}).
\end{align*}
Since both $y$ and $r$ are odd, the term in brackets is even and so $(y + 2^{k-1})^2 \equiv t \pmod {2^{k+1}}$. This completes the induction.

Next, we show that, for each $t  \in \mathbb{Z}_{2^k}$ with $t \equiv 1 \pmod 8$, there are at least 4 elements $y \in  \mathbb{Z}_{2^k}$ with $y^2 \equiv t \pmod {2^k}$. From above, we know there is at least one such $y$ value. The elements $2^k - y$, $2^{k-1} \pm y$ are all distinct in $\mathbb{Z}_{2^k}$ and their squares are $\equiv t \pmod {2^k}$. Therefore,
$$
\Big\lvert \Big\{ x \in \mathbb{Z}_{2^k}: x^2 \equiv t \pmod {2^k} \Big\} \Big\rvert  \geq 4.
$$
Since there are $2^{k-1}$ odd elements in $\mathbb{Z}_{2^k}$ and $2^{k-3}$ elements which are $\equiv 1 \pmod 8$, there can be no values of $t$ with 
$$
\Big\lvert \Big\{ x \in \mathbb{Z}_{2^k}: x^2 \equiv t \pmod {2^k} \Big\} \Big\rvert  > 4.
$$
The lemma follows.
\end{proof}

\bigskip

\begin{lemma}
\label{02k}
$N_2 (0, 2^k) = 2^k$ when  $k \geq 1$.
\end{lemma}
\begin{proof}
For any $x, y \in \mathbb{Z}_{2^k}$, 
$$
\ord_2(x^2 + y^2) = 
\begin{cases}
   \min \{ \ord_2(x^2), \ord_2(y^2) \} & \, \, \text{if } \, \, \ord_2(x^2) \ne \ord_2(y^2)  \\
   1 +  \min \{ \ord_2(x^2), \ord_2(y^2) \} & \, \, \text{if } \, \, \ord_2(x^2) = \ord_2(y^2). 
\end{cases}
$$
Suppose $(x, y) \in S_2(0, 2^k)$. 

If $k$ is even, then either $x = 0$ or $\ord_2(x) \geq k/2$. The same applies to $y$. Now,
$$
\{ x: \ord_2(x) \geq k/2 \} = \{ 2^{k/2} j : j = 1, 2, \dots , 2^{k-k/2} - 1 \}.
$$
Counting $x = 0$, there are therefore $2^{k/2}$ choices for $x$ and the same number for $y$. We then have $N_2(0, 2^k) = 2^{k/2} \times 2^{k/2} = 2^k$.

If $k$ is odd, then either $x, y \in \{ 2^{(k+1)/2} j: j = 0, 1, 2, \dots , 2^{k - (k+1)/2} - 1 \}$ or $\ord_2(x) = \ord_2(y) = (k-1)/2$. We have,
$$
\ord_2(x) = (k-1)/2 \iff  x \in  \{ 2^{(k-1)/2} j: j = 1, 3, 5, 2^{(k+1)/2} - 1 \}.
$$
So, when $k$ is odd,
$$
N_2(0, 2^k) = 2^{(k-1)/2} \times 2^{(k-1)/2} + 2^{(k-1)/2} \times 2^{(k-1)/2} = 2^k.
$$
\end{proof}

\bigskip

\begin{lemma}
\label{03k}
When $k \geq 1$, we have 
$$
N_{3}(0, 2^k) = 
\begin{cases}
    2^{3k/2}   & \text{if } \, \, k \text{ is even} \\
    2^{(3k + 1)/2}   & \text{if } \, \, k \text{ is odd}. \\
\end{cases}
$$
\end{lemma}
\begin{proof}
The result when $k < 3$ is clear so assume $k \geq 3$.
$$
\lvert \{ (x, y, z ): x^2 + y^2 + z^2 \equiv 0 \pmod {2^k} \} \rvert = \sum_{z = 0}^{2^k - 1} \lvert \{ (x, y): x^2 + y^2 \equiv -z^2 \pmod {2^k} \} \rvert.
$$
We break this sum into two parts. The first part contains values of $z$ for which $z^2 \equiv 0 \pmod {2^k}$. There are $2^{\lfloor k/2 \rfloor}$ values of $z$ such that $z^2 \equiv 0 \pmod {2^k}$. From lemma \ref{02k} $\lvert \{ (x, y ): x^2 + y^2 \equiv 0 \pmod {2^k} \} \rvert = 2^k$. The first part of the sum is therefore equal to $2^{k + \lfloor k/2 \rfloor}$.

The second part of the sum contains non-zero values of $z^2$. For such $z$ we can write $-z^2 = 2^{2r} s \pmod {2^k}$, where $2r < k$ and $s \equiv 7 \pmod 8$. For this value of $z$, we have from theorem \ref{m=2:p=2}
$$
N_{2}(-z^2, 2^k) = 
\begin{cases}
    2^{k}   & \, \, \text{if} \, \, 2r = k - 1 \\
    0 &  \, \, \text{ otherwise }.
\end{cases}
$$
If $k$ is even, no such $z$ values exist. If k is odd,
\begin{align*}
\lvert \{ & z: -z^2 = 2^{k-1}s \pmod {2^k} \, \, \text{ with } s \equiv 7 \pmod 8 \} \rvert \\
= \, \, & \lvert \{ z: z = 2^{(k-1)/2} u  \, \, \text{ with } \, \, u \, \, \text{ odd and } \, \, 1 \leq u < 2^{(k+1)/2} \} \rvert \\
= \, \, & 2^{(k-1)/2}.
\end{align*}
The second part of the sum is therefore equal to zero when $k$ is even and $2^k \times 2^{(k-1)/2} = 2^{(3k - 1)/2}$ when $k$ is odd. Combining the two parts of the sum produces the result.
\end{proof}

\bigskip

\begin{lemma}
\label{2s2k}
Let $t \in \mathbb{Z}_{2^k}$ with  $k \geq 1$  and $2 \nmid t$. Then 
$$
N_2 (t, 2^k) =
\begin{cases}
2 & \text{if } \, \, k = 1  \\
2^{k+1} &  \text{if } \, \, t \equiv 1 \pmod 4  \, \, \text{ and } \, \, k \geq 2 \\
0 &  \text{if } \, \, t \equiv 3 \pmod 4 \, \, \text{ and } \, \, k \geq 2.
\end{cases}
$$
\end{lemma}
\begin{proof}
We assume $k \geq 3$ since it is easy to see that $N_2(1, 2) = 2$, $N_2(1, 4) = 8$ and $N_2(3, 4) = 0$. If $t \in \{3, 7 \} \mod 8$, then there are no $x, y$ such that $x^2 + y^2 \equiv t \pmod 8$ and \textit{a fortiori} $N_2(t, 2^k) = 0$.

Suppose $t \equiv 1 \pmod 8$. If $(x, y) \in S_2(t, 2^k)$, then either $x^2 \equiv 1 \pmod 8$ and $y^2 \equiv 0 \pmod 8$ or  \textit{vice versa}. Now, 
$$
\lvert \{ y \in \mathbb{Z}_{2^k}: y^2 \equiv 0 \pmod 8 \} \rvert = \lvert \{ y \in \mathbb{Z}_{2^k}: y \equiv 0 \pmod 4 \} \rvert = 2^{k-2}.
$$
For each such $y$, $t - y^2 \equiv 1 \pmod {8}$.  By lemma \ref{x2=1mod8} there are $4$ values of $x \in  \mathbb{Z}_{2^k}$ satisfying the congruence $x^2 \equiv t - y^2 \pmod {2^k}$. Hence, when $t \equiv 1 \pmod 8$,
$$
N_2 (t, 2^k) = 4 \times 2^{k-2} + 2^{k-2} \times 4 = 2^{k+1}.
$$
Suppose finally that $t \equiv 5 \pmod 8$. If $(x, y) \in S_2(t, 2^k)$, then either $x^2 \equiv 1 \pmod 8$ and $y^2 \equiv 4 \pmod 8$ or  \textit{vice versa}. The same argument as above can be used again.\end{proof}

\bigskip

\begin{lemma}
\label{22s2k}
Let $t \in \mathbb{Z}_{2^k}$ with  $k \geq 2$  and $2 \nmid t$. Then 
$$
N_2 (2t, 2^k) =
\begin{cases}
4 & \, \, \text{if } \, \, k = 2  \\
2^{k+1} & \, \,  \text{if } \, \, t \equiv 1 \pmod 4  \, \,  \text{and }  k \geq 3\\
0 &  \, \, \text{if }  \, \, t \equiv 3 \pmod 4   \, \, \text{and }  k \geq 3.
\end{cases}
$$
\end{lemma}
\begin{proof}
We assume $k \geq 3$ since the result for $k = 2$ is straightforward. If $t \equiv 3 \pmod 4$ then $2t \equiv 6 \pmod 8$ and $N_2 (2t, 2^k) = 0$.

Suppose $t \equiv 1 \pmod 4$. Then $2t \equiv 2 \pmod 8$.  If $(x, y) \in S_2(t, 2^k)$, then both $x^2$ and $y^2$ must be $\equiv 1 \pmod 8$. Now, 
$$
\lvert \{ y \in \mathbb{Z}_{2^k}: y^2 \equiv 1 \pmod 8 \} \rvert = \lvert \{ y \in \mathbb{Z}_{2^k}: y \equiv 1 \pmod 2 \} \rvert = 2^{k-1}.
$$
For each such $y$, $t - y^2 \equiv 1 \pmod {8}$.  By lemma \ref{x2=1mod8} there are $4$ values of $x \in  \mathbb{Z}_{2^k}$ satisfying the congruence $x^2 \equiv t - y^2 \pmod {2^k}$. Hence, when $t \equiv 1 \pmod 4$,
$$
N_2 (2t, 2^k) = 4 \times 2^{k-1} = 2^{k+1}.
$$
\end{proof}

\bigskip

\begin{lemma}
\label{3s2k}
Let $t \in \mathbb{Z}_{2^k}$ with  $k \geq 2$  and $2 \nmid t$. Then 
$$
N_3 (t, 2^k) =
\begin{cases}
3 \times 2^{2k-1} &  \text{if } \, \, t \equiv 1 \pmod 4  \\
8 &  \text{if } \, \, t \equiv 3 \pmod 4 \, \, \text{ and } \, \, k = 2 \\
2^{2k} &  \text{if } \, \, t \equiv 3 \pmod 8 \, \, \text{ and } \, \, k \geq 3  \\
0 &  \text{if } \, \, t \equiv 7 \pmod 8.
\end{cases}
$$
\end{lemma}
\begin{proof}
The result is clear when $k = 2$ and when $t \equiv 7 \pmod 8$. Assume $k \geq 3$. If $x^2 + y^2 + z^2 \equiv t \pmod {2^k}$, then either all of $\{ x, y, z \}$ are odd, or two of $\{ x, y, z \}$ are even and the other odd. We will treat each possibility separately. We first count the number of suitable $\{x, y, z \}$ such that all are odd. In this case, we must have $t \equiv 3 \pmod 8$. The number of such solutions is:
$$
N_3 (t, 2^k)  =   \sum_{z \, \, \text{ odd }} \lvert \{ (x, y) \in \mathbb{Z}_{2^k}: x, y \, \, \text{ odd } \, \, \text{ and } \, \,  x^2 + y^2 \equiv t - z^2 \pmod {2^k} \} \rvert  
$$
Since  $t - z^2 \equiv 2 \pmod 8$ for odd $z$, we can write $z = 2u +1$ and use lemma \ref{22s2k} to show that:
$$
N_3 (t, 2^k)  = 2^{k + 1} \times \lvert \{ u:0 \leq u < 2^{k-1} \} \rvert = 2^{2k}.
$$

Secondly, we count the number of solutions in which exactly two of $\{ x, y, z \}$ are even. In this case, $t \equiv 1 \pmod 4$. The element $z$ is even in two out of the three possible arrangements of odd/even elements. Writing $z = 2u$, and using lemma \ref{2s2k}, the number of such solutions when $t \equiv 1 \pmod 4$ is
\begin{align*}
N_3 (t, 2^k)  &=   \frac{3}{2} \times \sum_{u = 0}^{2^{k-1}-1} \lvert \{ (x, y) \in \mathbb{Z}_{2^k}:  \, \,  x^2 + y^2 \equiv t - 4 u^2 \pmod {2^k} \} \rvert  \\
  &= \frac{3}{2} \times 2^{k+1} \times 2^{k-1} \\
  &= 3 \times 2^{2k-1}.
\end{align*}
\end{proof}

\bigskip

\begin{lemma}
\label{32s2k}
Let $t \in \mathbb{Z}_{2^k}$ with  $k \geq 2$  and $2 \nmid t$. Then $N_3 (2t, 2^k) = 3 \times 2^{2k - 1}$.
\end{lemma}
\begin{proof}
If $x^2 + y^2 + z^2 \equiv 2t \pmod {2^k}$, then two of $\{ x, y, z \}$ must be odd and the other even. So,
\begin{align*}
N_3 (2t, 2^k) \, \, =& \, \,  \lvert \{ (x, y, z) \in \mathbb{Z}_{2^k}: x^2 + y^2 + z^2 \equiv 2t \pmod {2^k} \} \rvert \\
  =& \, \, 3 \times \sum_{u = 0}^{2^{k-1} - 1} \lvert \{ (x, y) \in \mathbb{Z}_{2^k}: x^2 + y^2 \equiv 2(t - 2u^2) \pmod {2^k} \} \rvert .
\end{align*}
When $u$ is even, $t - 2u^2 \equiv t \pmod 4$ and when $u$ is odd, $t - 2u^2 \equiv t -2 \pmod 4$. By lemma \ref{22s2k}, we have,
\begin{align*}
N_3 (2t, 2^k) & =   3 \times \sum_{u \, \, \text{even}} \lvert \{ (x, y) \in \mathbb{Z}_{2^k}: x^2 + y^2 \equiv 2t \pmod {2^k} \} \rvert  \\
                 & + \, \, 3 \times \sum_{u \, \, \text{odd}} \lvert \{ (x, y) \in \mathbb{Z}_{2^k}: x^2 + y^2 \equiv 2(t -2) \pmod {2^k} \} \rvert \\
                 & = 3 \times  2^{k-2} \times  2^{k+1} \\
                 & = 3 \times 2^{2k - 1}.
\end{align*}
\end{proof}

\bigskip

\begin{lemma}
\label{2kred2k-2}
Let $t \in \mathbb{Z}_{2^k}$ with  $k \geq 3$ and $t = 2^r s$ with $2 \nmid s$. If $r \geq 2$, then $N_2 (t, 2^k) = 4  \times  N_2 (2^{r-2} s, 2^{k-2})$ and  $N_3 (t, 2^k) = 8  \times  N_3 (2^{r-2} s, 2^{k-2})$.
\end{lemma}
\begin{proof}
We introduce the map $f:   N_2 (2^{r} s, 2^{k}) \to N_2 (2^{r-2} s, 2^{k-2})$ given by $f((x, y)) = (x/2 \pmod {2^{r-2}}, y/2 \pmod {2^{r-2}})$. The map is well defined because, if $(x, y) \in S_2(2^{r} s, 2^{k})$ for $k \geq3$ and $r \geq 2$, then both $x$ and $y$ must be even. The map is onto since, if $(x, y) \in S_2(2^{r-2} s, 2^{k-2})$, then  $(2x, 2y) \in S_2(2^{r} s, 2^{k})$ and $f((2x, 2y)) = (x, y)$. The map is also $4$ to $1$. If $(x, y) \in  S_2(2^{r} s, 2^{k})$, then so are $(x + 2^{r-1}, y),  (x, y + 2^{r-1})$ and  $(x + 2^{r-1}, y + 2^{r-1})$ and all four elements are mapped to the same point by $f$. If $f((x, y)) = f((u, v))$ then $u/2 \equiv x/2 \pmod {2^{r-2}}$ so $u \equiv x \pmod {2^{r-1}}$. Therefore, $u \pmod {2^{r}} \in \{x, x + 2^{r-1} \}$. Similarly, $v \pmod {2^{r}} \in \{y, y + 2^{r-1} \}$. Therefore, $(u, v)$ must be one of the four elements $(x, y), (x + 2^{r-1}, y),  (x, y + 2^{r-1})$ or  $(x + 2^{r-1}, y + 2{r-1})$. The lemma follows.

With suitable changes, the same argument works for $N_3 (t, 2^k)$.
\end{proof}

\bigskip

\subsection{ $N^*_m(t, 2^k)$ : Proof of theorem \ref{Nstar:m:p=2}}

We note that theorem \ref{Nstar:m:p=2} appears as Theorem 6.2 in \cite{Mollahajiaghaei_2017} and Theorem 4.4 in \cite{LI201841}.

The proof when $k \in \{1, 2 \}$ is obvious so assume $k \geq 3$. When $x$ is odd, $x^2 \equiv 1 \pmod 8$. Therefore, there are no solutions when $t \not\equiv m \pmod 8$. 

The proof is by induction on $m$. The theorem holds for $m = 1$ by Lemma \ref{x2=1mod8}. Suppose the theorem holds for $m - 1$. Then, if $t \equiv m \pmod 8$,  
\begin{align*}
  N^*_m(t, 2^k) =& \, \, \Big\lvert \Big\{  (x_1, x_2, ... , x_m)  \in \mathbb{Z}^*_{2^k}: \sum_{i=1}^{m} x_i ^2 \equiv t \pmod {2^k} \Big\} \Big\rvert \\
     =& \, \, \sum_s N^*_{m-1} (s, 2^k) \times N^*_1 (t-s, 2^k) \\
     =& \, \, 4 \times \sum_{s \equiv m - 1 \pmod 8} N^*_{m-1} (s, 2^k) \, \, \text{ by Lemma \ref{x2=1mod8}}  \\
     =& \, \, 4 \times 2^{k-3} \times 2^{2(m-1) + (m-2)(k-3)}  \, \, \text{ by the inductive step} \\
     =& \, \, 2^{2m + (m-1)(k-3)}.
\end{align*}

This finishes the proof.

\bigskip

\subsection{ $N_m(t, 2)$ and $N_m(t, 4)$ : Proof of theorem \ref{m:p=2}}

Firstly, if $t, x_i \in \mathbb{Z}_{2}$ for $i = 1, 2, \dots , m$, then 
\begin{align*}
 \sum_{i=1}^{m} x_i ^2 \equiv t \pmod 2 \iff & \sum_{i=1}^{m} x_i \equiv t \pmod 2 \\
     \iff &  \lvert \{  x_i : x_i = 1  \} \rvert   \equiv t \pmod 2.
\end{align*}
  So,
  $$
  N_m(t, 2) = \sum_j \binom{m}{t + 2j}.
  $$
The result for $N_m(t, 2)$ follows from the following identities, which are special cases of \eqref{eqn:Ramus}:
$$
\sum_j \binom{m}{2j} = \sum_j \binom{m} {1 + 2j} = 2^{m-1}.
$$

\bigskip
Similarly, if $t, x_i \in \mathbb{Z}_{4}$ for $i = 1, 2, \dots , m$, then 
$$
 \sum_{i=1}^{m} x_i ^2 \equiv t \pmod 4 \iff \lvert \{  x_i : x_i  \equiv 1 \pmod 2 \} \rvert   \equiv t \pmod 4
$$
 So,
$$
 N_m(t, 4) \, \, = \, \, \sum_j \binom{m}{t + 4j} 2^{t + 4j}  \,\, 2^{m - t - 4j} \, \, = \, \, 2^m \sum_j \binom{m}{t + 4j}.
$$

The result for  $N_m(t, 4)$ follows from another application of \eqref{eqn:Ramus}.

\bigskip

\subsection{ $N_2(t, 2^k)$ : Proof of theorem \ref{m=2:p=2}}

The result is clear when $k = 2$ so assume $k \geq 3$.

The $t = 0$ case is covered by lemma \ref{02k}. For non-zero $t$, write $t = 2^{\alpha}\beta$ where $2 \nmid \beta$. We use proof by descent.  By lemma \ref{2kred2k-2},  $N_2 (t, 2^k) = 4  \times  N_2 (2^{\alpha-2} \beta, 2^{k-2})$ when $\alpha \geq 2$. Continuing the reduction, we get
$$
N_2 (t, 2^k) =
\begin{cases}
   4^{\alpha/2} \times N_2(\beta, 2^{k - \alpha}) & \text{if } \, \, \alpha  \equiv 0 \pmod 2  \\
   4^{(\alpha - 1)/2} \times N_2(2 \beta, 2^{k - \alpha + 1}) & \text{if } \, \, \alpha  \equiv 1 \pmod 2. 
\end{cases}
$$

We can then use lemmas \ref{2s2k} and \ref{22s2k} to derive the expressions for $N_2 (t, 2^k)$ given in the theorem.

\bigskip

\subsection{ $N_3(t, 2^k)$ : Proof of theorem \ref{m=3:p=2}}

The proof of theorem \ref{m=3:p=2} follows the same pattern as the proof of theorem  \ref{m=2:p=2}. The result for $t = 0$ is contained in lemma \ref{03k}. For non-zero $t$, write $t = 2^{\alpha}\beta$ where $2 \nmid \beta$. We again use proof by descent.  By lemma \ref{2kred2k-2},  $N_3 (t, 2^k) = 8  \times  N_3 (2^{\alpha-2} \beta, 2^{k-2})$ when $\alpha \geq 2$. Continuing the reduction, we get
$$
N_3 (t, 2^k) =
\begin{cases}
   8^{\alpha/2} \times N_3(\beta, 2^{k - \alpha}) & \text{if } \, \, \alpha  \equiv 0 \pmod 2  \\
   8^{(\alpha - 1)/2} \times N_3(2 \beta, 2^{k - \alpha + 1}) & \text{if } \, \, \alpha  \equiv 1 \pmod 2. 
\end{cases}
$$

We can then use lemmas \ref{3s2k} and \ref{32s2k} to derive the expressions for $N_3 (t, 2^k)$ given in the theorem.

\bigskip

\section{The $\mathbb{Z}_p$ case: proof of theorems \ref{k=1:m odd:p=1mod4} - \ref{k=1:m even:p=3mod4}}
\label{Z_p} 

In this section we will prove the four theorems \ref{k=1:m odd:p=1mod4}, \ref{k=1:m even:p=1mod4}, \ref{k=1:m odd:p=3mod4} and \ref{k=1:m even:p=3mod4} related to $\mathbb{Z}_p$. We first prove the theorems for $N_1$ and $N_2$ and then establish a recurrence relation which can be used to derive the general formulae. Using the definition of $\gamma$ given in (\ref{gamma}), it is easy to see that
\begin{equation}
\label{gamma1}
\gamma_{1} = \begin{pmatrix} 1 \\ 0 \\ 2 \end{pmatrix}
\end{equation}
which establishes theorems \ref{k=1:m odd:p=1mod4} and \ref{k=1:m odd:p=3mod4} for $N_1$.

\bigskip

We include the proof of the case $m=2$ here even though it is available elsewhere. It requires a few lemmata. 

\bigskip

\begin{lemma}
\label{y2=x2+t}
For each non-zero $t \in \mathbb{Z}_p$ the congruence $y^2 = x^2 + t \pmod p$ has $p-1$ solutions $(x,y)$.
\end{lemma}
\begin{proof}
Rearranging we have $y^2 - x^2 = t \pmod p$. Factorising the LHS and changing variables to $u = y-x$, $v = y+x$, which is an invertible map in $\mathbb{Z}_p$, we have $uv = t \pmod p$. For each non-zero choice for $\, \, u \pmod p$ there is a unique value for $v = t*u^{-1} \pmod p$ which  satisfies the congruence. There are therefore $p-1$ solutions for $(u, v)$ and the number of solutions of the original congruence in terms of the variables $(x, y)$ is the same as the number of solutions for $(u, v)$.
\end{proof}

\bigskip

\begin{lemma}
\label{sumx2+tp}
For fixed non-zero $t \in \mathbb{Z}_p$, $\sum_{x=0}^{p-1} \Legendre{x^2 + t}{p} = -1$.
\end{lemma}
\begin{proof}
On the one hand the lemma \ref{y2=x2+t} says the number of solutions to the congruence $y^2 = x^2 + t \pmod p$ is $p-1$. On the other hand for each fixed $x$ the number of $y$ satisfying $y^2 = x^2 + t \pmod p$ is $1 + \Legendre{x^2 + t}{p}$. Summing this number over the $p$ $x$-values and equating the result to $p-1$ produces the required result.
\end{proof}

\bigskip

\begin{corollary}
\label{k=1:m=2}
The number of $(x, y)$ solutions to the congruence $x^2 + y^2 = 0 \pmod p$ is
$$
\begin{cases}
    2p - 1 & \text{if } \, \, p = 1 \pmod 4 \\
    1 & \text{if} \, \, \, p = 3 \pmod 4.
\end{cases}
$$
For each non-zero $t \in \mathbb{Z}_p$, the number of $(x, y)$ solutions to the congruence $x^2 + y^2 = t \pmod p$ is
$$
\begin{cases}
    p - 1 & \text{if } \, \, p = 1 \pmod 4 \\
    p + 1 & \text{if} \, \, \, p = 3 \pmod 4.
\end{cases}
$$
\end{corollary}
\begin{proof}
Firstly, $-1$ is a quadratic residue $\pmod p \iff p = 1 \pmod 4$. Therefore, $x^2 = -y^2$ has no solution if $p = 3 \pmod 4$ other than $(0, 0)$ and has \mbox{$2(p-1) + 1 = 2p - 1$} solutions when $p = 1 \pmod 4$.
Next, if $t \neq 0$, the number of $(x, y)$ satisfying $x^2 + y^2 = t \pmod p$ is
$$
\sum_{x = 0}^{p-1} (1 + \Legendre{t - x^2}{p})    
$$
$$
= p + \sum_{x = 0}^{p-1} \Legendre{t - x^2}{p}
$$
$$
= p + \Legendre{-1}{p} \sum_{x = 0}^{p-1} \Legendre{x^2 - t}{p}. 
$$
The corollary follows from Lemma \ref{sumx2+tp}.
\end{proof}

\bigskip

In terms of the vector $\gamma$, Corollary \ref{k=1:m=2} says that when $p = 1 \pmod 4$,
\begin{equation}
\label{gamma2-1mod4}
\gamma_{2} = \begin{pmatrix}  2p - 1\\ p - 1 \\ p - 1 \end{pmatrix}
\end{equation}
and when $p = 3 \pmod 4$,
\begin{equation}
\label{gamma2-3mod4}
\gamma_{2} = \begin{pmatrix}  1\\ p + 1 \\ p + 1 \end{pmatrix}.
\end{equation}
 
\bigskip
 
This establishes theorems \ref{k=1:m even:p=1mod4} and \ref{k=1:m even:p=3mod4} for $N_2$.

In order to establish a recursive formula for the vector $\gamma_m$ in terms of $\gamma_{m-1}$  we need two preliminary results. 

When $t$ is a non-residue mod $p$ we have:

\bigskip

\begin{lemma}
\label{t-z^2modp1}
Suppose $\Legendre{t}{p} = -1$. Then
$$
\lvert \Big\{ z: \Legendre{t - z^2}{p} = -1 \Big\} \rvert =
\begin{cases}
    \frac{p+1}{2}  & \text{if } \, \, p = 1 \pmod 4 \\
    \frac{p-1}{2} & \text{if} \, \, p = 3 \pmod 4
\end{cases}
$$
and
$$
\lvert \Big\{ z: \Legendre{t - z^2}{p} = 1 \Big\} \rvert =
\begin{cases}
    \frac{p-1}{2}  & \text{if } \, \, p = 1 \pmod 4 \\
    \frac{p+1}{2} & \text{if} \, \, \, p = 3 \pmod 4
\end{cases}
$$
\end{lemma}
\begin{proof}
Since $t$ is a non-residue, $t - z^2 \neq 0 \pmod p$ for all $z$. From Corollary \ref{k=1:m=2}, the equation $x^2 + z^2 = t$ has $p-1$ solutions $(x, z)$ if $p = 1 \pmod 4$ and $p+1$ solutions $(x, z)$ if $p = 3 \pmod 4$. Therefore, $t - z^2$ is a quadratic residue $\pmod p$ for $\frac{p-1}{2}$ values of $z$ if $p = 1 \pmod 4$ and for $\frac{p+1}{2}$ values of $z$ if $p = 3 \pmod 4$. Since $t - z^2 \neq 0 \pmod p$, $t - z^2$ is a non-quadratic residue $\pmod p$ for the remaining values of $z$. 
\end{proof}
\bigskip

When $t$ is a residue mod $p$ we have:

\bigskip

\begin{lemma}
\label{t-z^2modp2}
Suppose $\Legendre{t}{p} = 1$. Then 
$$
\lvert \Big\{ z: \Legendre{t - z^2}{p} = -1 \Big\} \rvert =
\begin{cases}
    \frac{p-1}{2}  & \text{if } \, \, p = 1 \pmod 4 \\
    \frac{p-3}{2} & \text{if} \, \, p = 3 \pmod 4
\end{cases}
$$
and
$$
\lvert \Big\{ z: \Legendre{t - z^2}{p} = 1 \Big\} \rvert = 
\begin{cases}
    \frac{p-3}{2}  & \text{if } \, \, p = 1 \pmod 4 \\
    \frac{p-1}{2} & \text{if} \, \, p = 3 \pmod 4
\end{cases}
$$
\end{lemma}
\begin{proof}
Since $t$ is a residue, $t - z^2 = 0 \pmod p$ when $z = \pm \sqrt{t}$. When $p = 1 \pmod 4$, Corollary \ref{k=1:m=2} says there are $p-1$ solutions $(x, z)$ to the congruence $x^2 + z^2 = t$ and these solutions can be written as:
$$
(0, \pm \sqrt{t}), (\pm \sqrt{t}, 0), (x_i, z_i)_{i = 1}^{p-5}
$$
where the solutions are arranged so that $z_{2i-1} = z_{2i}$ for all $i$ (i.e $x_{2i-1} = -x_{2i}$). Excluding the solutions $(0, \pm \sqrt{t})$ and including $z = 0$, we see that $t-z^2$ is a quadratic residue for $1 + \frac{p-5}{2} = \frac{p-3}{2}$ values of $z$ when $p = 1 \pmod 4$. A similar argument can be used when $p = 3 \pmod 4$.
\end{proof}

\bigskip

Let $t \in \mathbb{Z}_{p}$. By writing the congruence $\sum_{i=1}^{m+1} {x_i}^2 = t \, \,$ as $\, \, \sum_{i=1}^{m} {x_i}^2 = t - x_{m+1}^2$, we have:
 
\begin{align}
\label{recurrence1} 
\begin{split}
N_{m+1}(t, p) & =  \sum_{s \in \mathbb{Z}_p}  \lvert \Big\{ (x_i)_{i = 1}^{m}: \sum_{i=1}^{m} {x_i}^2 = s \Big\} \rvert \times \lvert \Big\{ x_{m+1}: t - x_{m+1}^2 = s \Big\} \rvert \\
                 & = \lvert \Big\{ (x_i)_{i = 1}^{m}: \sum_{i=1}^{m} {x_i}^2 = 0 \Big\} \rvert \times \lvert \Big\{ x_{m+1}: t - x_{m+1}^2 = 0 \Big\} \rvert \\
                 & + \, \, \lvert \Big\{ (x_i)_{i = 1}^{m}:  \sum_{i=1}^{m} {x_i}^2 = s_1 \Big\} \rvert \times \lvert \Big\{ x_{m+1}: \Legendre{t - x_{m+1}^2}{p} = -1  \Big\} \rvert \\ 
                 & + \, \, \lvert \Big\{ (x_i)_{i = 1}^{m}:  \sum_{i=1}^{m} {x_i}^2 = s_2 \Big\} \rvert \times \lvert \Big\{ x_{m+1}: \Legendre{t - x_{m+1}^2}{p} = 1  \Big\} \rvert
\end{split}
\end{align}

where $s_1$ is a fixed non-residue $\pmod p$ and $s_2$ is a fixed residue $\pmod p$ (the particular choices do not matter due to Lemma \ref{resornotmodp}). We can rewrite equation (\ref{recurrence1}) as
\begin{align}
\label{recurrence2}
\begin{split}
N_{m+1}(t, p) ={}& \gamma_{m,1} \times \lvert \{ x_{m+1}: t - x_{m+1}^2 = 0 \} \rvert \, \, + \\
                     & \gamma_{m, 2} \times \lvert \{ x_{m+1}: \Legendre{t - x_{m+1}^2}{p} = -1  \} \rvert \, \, + \\
                     & \gamma_{m, 3} \times \lvert \{ x_{m+1}: \Legendre{t - x_{m+1}^2}{p} = 1  \} \rvert.
\end{split}
\end{align}
 
By substituting in turn $t = 0$, $t$ a non-residue and $t$ a residue $\pmod p$ in equation (\ref{recurrence2}), we obtain three linear equations for the components of $\gamma_{m+1}$ in terms of the components of $\gamma_m$. The resulting matrix equation can then be solved. Obviously, 
$$ 
\lvert \{ x_{m+1}:  - x_{m+1}^2 = 0 \} \rvert = 1
$$
and $\lvert \{ x_{m+1}: \Legendre{- x_{m+1}^2}{p} = \pm 1 \} \rvert$ is either $0$ or $p-1$ depending on $p \pmod 4$. We can obtain $\lvert \{ x_{m+1}: \Legendre{t - x_{m+1}^2}{p} = -1  \} \rvert$ and $ \lvert \{ x_{m+1}: \Legendre{t - x_{m+1}^2}{p} = 1  \} \rvert$ from Lemma \ref{t-z^2modp1} and Lemma \ref {t-z^2modp2}. The resulting equations can be represented in matrix form as $\gamma_{m+1}  =  A \times \gamma_m$ when $p = 1 \pmod 4$, and $\gamma_{m+1}  =  B \times \gamma_m$ when $p = 3 \pmod 4$, where the matrices $A$ and $B$ are defined by:
\begin{align}
\label{gammam-mod4}
A  := \begin{pmatrix}  1, 0, p - 1\\ 0, \frac{p + 1}{2}, \frac{p-1}{2} \\ 2, \frac{p - 1}{2}, \frac{p-3}{2} \end{pmatrix}
\, \, \text{ and } \, \, 
B := \begin{pmatrix}  1, p - 1, 0 \\ 0, \frac{p - 1}{2}, \frac{p + 1}{2} \\ 2, \frac{p - 3}{2}, \frac{p - 1}{2} \end{pmatrix}
.
\end{align}

\bigskip

Therefore,
$$
\gamma_m = A^{m-1} \times \gamma_1 \, \, \text{ if p = 1 mod 4}
$$
and
$$
\gamma_m = B^{m-1} \times \gamma_1  \, \, \text{ if p = 3 mod 4}.  
$$

\bigskip

Theorems \ref{k=1:m odd:p=1mod4} and \ref{k=1:m even:p=1mod4} now follow from the form of $\gamma_1$ and $\gamma_2$ and the identities:
$$
A^2 \times \begin{pmatrix}  p^{2m} \\ p^{2m} - p^m \\ p^{2m} + p^m \end{pmatrix}
=
\begin{pmatrix}  p^{2m+2} \\ p^{2m+2} - p^{m+1} \\ p^{2m+2} + p^{m+1} \end{pmatrix}
\, \, \text{for} \, \, m \geq 0, 
$$
$$
A^2 \times \begin{pmatrix}  p^{2m-1} + p^{m} - p^{m-1} \\ p^{2m-1} - p^{m-1} \\ p^{2m-1 } - p^{m-1} \end{pmatrix}
=
\begin{pmatrix}  p^{2m+1} + p^{m+1} - p^{m} \\ p^{2m+1} - p^{m} \\ p^{2m+1} - p^{m} \end{pmatrix}
\, \, \text{for} \, \, m \geq 1.
$$

\bigskip 

Similarly, Theorems \ref{k=1:m odd:p=3mod4} and \ref{k=1:m even:p=3mod4} follow from the identities:

\bigskip

$$
B^2 \times \begin{pmatrix}  p^{2m} \\ p^{2m} + (-1)^{m+1} p^m \\ p^{2m} + (-1)^m p^m \end{pmatrix}
=
\begin{pmatrix}  p^{2m+2} \\ p^{2m+2} + (-1)^{m} p^{m+1} \\ p^{2m+2} + (-1)^{m+1} p^{m+1} \end{pmatrix}
\, \, \text{for} \, \, m \geq 0, 
$$
and
\begin{align*}
\begin{split}
B^2 \times &\begin{pmatrix}  p^{2m-1} + (-1)^m p^{m} + (-1)^{m-1} p^{m-1} \\ p^{2m-1} + (-1)^{m-1} p^{m-1} \\ p^{2m-1 } + (-1)^{m-1} p^{m-1} \end{pmatrix} \\
\\
 = & \begin{pmatrix}  p^{2m+1} + (-1)^{m+1} p^{m+1} + (-1)^{m} p^{m} \\ p^{2m+1} + (-1)^m p^{m} \\ p^{2m+1} + (-1)^m p^{m} \end{pmatrix} \, \, \text{for} \, \, m \geq 1.
\end{split}
\end{align*}

\bigskip

\section{Sum of two squares in $\mathbb{Z}_{p^k}$: $N_2(t, p^k)$}
\label{twosquares}

\bigskip

\subsection{Proof of theorem \ref{m=2:p=1mod4}: $p = 1 \pmod 4$} 

\bigskip
We first deal with primes $p: p = 1 \pmod 4$ and calculate $N_2$ in $\mathbb{Z}_{p^k}$ where $k \geq 1$.

Since $\Legendre{-1}{p} = 1$, $\sqrt{-1}$ exists in $\mathbb{Z}_{p^k}$ by Hensel's Lemma. We make the invertible transformation:
$$
u = x + y \, \sqrt{- 1} \, , \, \, \, \, v = x - y \, \sqrt{- 1}.
$$
Then, for $t \in \mathbb{Z}_{p^k}$:
$$
\lvert \{ (x, y): x^2 + y^2 = t \pmod {p^k} \} \rvert \, \, = \, \, \lvert \{ (u, v): uv = t \pmod {p^k} \} \rvert.
$$

Let $t = 0$. Then, $uv = 0$ if and only if $u = 0$ or $v = 0$ or there is an $0 \leq \alpha < k$ such that $p^{\alpha} \| u$ and  $p^{k-\alpha} | v$. We have 
$$
\lvert \{u: p^{\alpha} \| u \} \rvert = p^{k-\alpha-1}(p-1) \, \, \text{ and } \, \, \lvert \{v: p^{k - \alpha} | v \} \rvert = p^{\alpha} - 1. 
$$
So,
\begin{align*}
\lvert \{ (u, v): uv = 0 \pmod {p^k} \} \rvert & = 2p^k - 1 + \sum_{\alpha = 0}^{k-1} p^{k-\alpha-1}(p^{\alpha}-1)(p-1) \\
                                                     & = p^{k-1} \left( p(k+1) - k \right).
\end{align*}

Next, assume $p^{\alpha} \| t$ and write $t = p^{\alpha}\beta$ where $0 \leq \alpha < k$ and $p \nmid \beta$. Then, $uv = t$ if and only if $u = p^{\theta} \delta$ and $v = p^{\alpha-\theta}\beta \delta^{-1} + n p^{k - \theta}$ for some $0 \leq \theta \leq \alpha$, $0 \leq n < p^{k - \theta}$ and $1 \leq \delta < p^{k - \theta}$ with $p \nmid \delta$. Then,
\begin{align*}
\lvert \{ (u, v): uv = t \} \rvert &= \sum_{\theta = 0}^{\alpha} \lvert \Big \{ \delta \in \mathbb{Z}_{p^k}: 1 \leq \delta < p^{k - \theta} \, \, \text{ and } \, \, p \nmid \delta \Big \} \times p^{\theta} \rvert \\
                       & = \sum_{\theta = 0}^{\alpha} p^{k - \theta - 1} (p - 1) p^{\theta} \\
                       & = \left( \alpha + 1 \right) (p-1) p^{k-1}.
\end{align*}

This completes the proof of theorem \ref{m=2:p=1mod4}.

\bigskip

\subsection{Proof of theorem \ref{m=2:p=3mod4}: $p = 3 \pmod 4$}

\bigskip

In this section we assume $p = 3 \pmod 4$ and determine $N_2$ in $\mathbb{Z}_{p^k}$ where $k \geq 1$.

We will first look at $N_2(0, p^k)$. When $p = 3 \pmod 4$, there are no non-trivial solutions $( x, y )$ to $x^2 + y^2 = 0 \pmod p$. Therefore, 
\begin{equation}
\label{ord}
ord_p(x^2 + y^2) = \min \{ord_p(x^2), ord_p(y^2) \}.
\end{equation}

So if $x^2 + y^2 = 0 \pmod {p^k}$, both $x^2$ and $y^2$ must be zero $\pmod {p^k}$. We have,
$$
\{x: x^2 = 0 \pmod {p^k} \} = \{ 0, p^{\lceil \frac{k}{2} \rceil}, 2 p^{\lceil \frac{k}{2} \rceil}, \dots , (p^{k-\lceil \frac{k}{2} \rceil}-1) p^{\lceil \frac{k}{2} \rceil} \}.
$$ 
Hence,
$$
\lvert \{x: x^2 = 0 \pmod {p^k} \} \rvert = p^{k-\lceil \frac{k}{2} \rceil} = p^{\lfloor \frac{k}{2} \rfloor}
$$
and the same result holds for $y$. Therefore, $N_2(0, p^k ) = p^{2 \lfloor \frac{k}{2} \rfloor}$.

Next let $t \in \mathbb{Z}_{p^k}$ and assume $p^{2\alpha + 1} \| t$ for some $\alpha: 1 \leq 2\alpha + 1 < k$. Then $x^2 + y^2 \neq t \pmod {p^k}$  from (\ref{ord}). Hence, $N_2(t, p^k ) = 0$ in this case.

Next assume $p \nmid t$. Any pair $(x, y)$ satisfying $x^2 + y^2 = t \pmod {p^k}$  is of the form $x = x_1 + mp^{k-1}$, $y = y_1 + np^{k-1}$ where ${x_1}^2 + {y_1}^2 = t \pmod {p^{k-1}}$ and $0 \leq m,n < p$. Then ${x_1}^2 + {y_1}^2 = t + rp^{k-1}$ for some $r$. Expanding the initial congruence we find that $p$ must divide $2 x_1 m + 2 y_1 n + r$. Since $p \nmid t$, one of $x_1$ or $y_1$ (say $x_1$) is not divisible by $p$ and is thus invertible in $\mathbb{Z}_p$. If $y_1 = 0$, then $m = -\frac{r}{2x_1} \pmod p$ and $n = \{0, 1, \dots , p-1 \}$. If $y_1 \neq 0$, $m = \frac{-r - 2 y_1 n}{2 x_1} \pmod p$. In either case, there are $p$ choices for the pair $(m, n)$ producing a solution to the original congruence. Therefore,
$$
\lvert \{ (x, y): x^2 + y^2 = t \pmod {p^k} \} \rvert = p \times \lvert \{ (x, y): x^2 + y^2 = t \pmod {p^{k-1}} \} \rvert.
$$
From theorem \ref{k=1:m even:p=3mod4}, $\lvert \{ (x, y): x^2 + y^2 = t \pmod p \} \rvert = p + 1$. So, when $p \nmid t$,

\begin{equation}
\label{r2t}
\lvert \{ (x, y): x^2 + y^2 = t \pmod {p^k} \} \rvert = (p+1)p^{k-1}
\end{equation}

We now assume $p^{2\alpha} \| t$ for some $\alpha: 0 < 2\alpha < k$ and write $t = p^{2\alpha}\beta$ where $p \nmid \beta$. If, $x^2 + y^2 = t \pmod {p^k}$, then $x^2 + y^2 = 0 \pmod p$ and so $x, y = 0 \pmod p$. Putting $x = px_1$, $y = py_1$ and dividing the congruence through by $p^2$ we have ${x_1}^2 + {y_1}^2 = p^{2\alpha - 2}\beta \pmod {p^{k-2}}$. Continuing in this way, we find $m, n \in \mathbb{Z}_{p^{k-\alpha}}$ such that $x = p^{\alpha}m$, $y = p^{\alpha}n$ and $m^2 + n^2 = \beta \pmod{p^{k - 2 \alpha}}$. From (\ref{r2t}) there are $(p+1)p^{k - 2\alpha-1}$ solutions to the congruence $a^2 + b^2 = \beta \pmod {p^{k-2\alpha}}$. Each of these solutions $(a, b)$ generates $p^{\alpha}$ values for $m, n \in \mathbb{Z}_{p^{k-\alpha}}$ given by
$$
m \in \{ a, a + p^{k-2\alpha}, \dots , a + (p^{\alpha} - 1)p^{k - 2 \alpha} \},
$$
$$
n \in \{ b, b + p^{k - 2 \alpha},  \dots ,  b + (p^{\alpha} - 1)p^{k - 2 \alpha} \}.
$$
The number of solutions to $x^2 + y^2 = p^{\alpha} \beta \pmod {p^k}$ is 
$$
(p+1)p^{k - 2\alpha-1} \times p^{\alpha} \times p^{\alpha} = (p+1) p^k.
$$
This completes the proof of theorem \ref{m=2:p=3mod4}.

\bigskip

\section{Sum of three squares in $\mathbb{Z}_{p^k}$: $N_3(t, p^k)$}
\label{threesquares}

\subsection{Some preliminary lemmas}

In this section we will calculate $N_3$ in $\mathbb{Z}_{p^k}$ where $k \geq 1$. We will use the results for $r_2$ and the decomposition:
\begin{equation}
\label{recurrencek} 
N_{3}(t, p^k ) =  \sum_{s \in \mathbb{Z}_{p^k}} N_2(s, p^k ) \times \lvert \Big\{ z: t - z^2 = s \Big\} \rvert.
\end{equation}

\bigskip

Let $t = p^{\alpha}\beta \in \mathbb{Z}_{p^k}$ where $p \nmid \beta$. The following lemmas give the number of $z \in \mathbb{Z}_{p^k} \, \, $ satisfying $\, \, p^{\gamma} \| t - z^2 \, \,$ for each $\gamma: 0 \leq \gamma < k$. 

\bigskip

\begin{lemma}
\label{lemma0}
$\lvert \Big\{ z: z^2 = 0 \pmod {p^k} \Big\} \rvert = p^{\lfloor k/2 \rfloor}.$
\end{lemma}
\begin{proof}
$z^2 = 0 \, \, $ if and only if $z = 0$ or $z = s p^{\lceil k/2 \rceil}$ for some $ \, \, s: 1 \leq s < p^{k - \lceil k/2 \rceil} - 1$.
\end{proof}

\bigskip

\begin{lemma}
\label{lemma1}
If $t \neq 0$,
$$
\lvert \Big\{ z: t - z^2 = 0 \pmod {p^k} \Big\} \rvert =
\begin{cases}
    2 p^{\alpha/2} & \text{if } \, \, \alpha \text{ is even and } \Legendre{\beta}{p} = 1 \\
    0                     & \text{ otherwise}.
\end{cases}
$$
\end{lemma}
\begin{proof}
If $\Legendre{t}{p} = 1$, $\sqrt{t}$ exists in $\mathbb{Z}_{p^k}$. Then $t - z^2 = 0 \pmod {p^k} \, \,$ if and only if $z = \pm \sqrt{t} + sp^{k - \alpha/2}$ for some $ \, \, s: 0 \leq s < p^{\alpha/2} - 1$. 
\end{proof}

\bigskip

\begin{lemma}
\label{lemma2}
If $\alpha$ is odd and $\gamma:0 \leq \gamma < k$ then
$$
\lvert \Big\{ z: p^{\gamma} \| t - z^2 \pmod {p^k} \Big\} \rvert =
\begin{cases}
    (p - 1)p^{k - 1 - \gamma/2} & \text{if } \, \, \gamma \text{ is even and } 0 \leq \gamma < \alpha \\
    p^{k - (\alpha + 1)/2} & \text{if } \, \, \gamma = \alpha \\
     0                                        & \text{ otherwise}.
\end{cases}
$$
\end{lemma}
\begin{proof}
If $z = 0 \, \,$, $p^{\alpha} \| t - z^2$. Suppose $z \neq 0$ with $p^{\delta} \| z$. Then,
\begin{align}
\begin{split}
\label{lem2}
p^{2\delta} \| t - z^2& \, \, \text{ for } 0 \leq \delta < \frac{\alpha}{2} \\
p^{\alpha} \| t - z^2& \, \, \text{ for } \frac{\alpha}{2} < \delta < k.
\end{split}
\end{align}
The statement for $\gamma$ even and less than $\alpha$ then follows from 
$$ 
\lvert \{ z: p^{\delta} \| z \} \rvert = (p-1) p^{k - 1 - \delta}.
$$ 
The statement for $\gamma = \alpha$ follows from
\begin{align*}
\lvert \{ z: p^{\delta} \| z  \, \, \text{ for } \frac{\alpha}{2} < \delta < k \} \rvert &= \sum_{\delta = (\alpha + 1)/2}^{k-1}(p-1) p^{k - 1 - \delta} \\
          & = \, \, p^{k - (\alpha + 1)/2} - 1.
\end{align*}
If $\gamma$ is odd but not equal to $\alpha$, it is clear from (\ref{lem2}) that $ord_p( t - z^2) \neq \gamma$.
\end{proof}

\bigskip

\begin{lemma}
\label{lemma3}
If $\alpha$ is even, $\Legendre{\beta}{p} = -1$ and $\gamma:0 \leq \gamma < k$ then
$$
\lvert \Big\{ z: p^{\gamma} \| t - z^2 \pmod {p^k} \Big\} \rvert =
\begin{cases}
    (p - 1)p^{k - 1 - \gamma/2} & \text{if } \, \, \gamma \text{ is even and } 0 \leq \gamma < \alpha \\
    0                                         & \text{if } \, \, \gamma \text{ is odd }  \\
    p^{k - \alpha/2} & \text{if } \, \, \gamma = \alpha \\
    0       & \text{ if }\, \,  \alpha < \gamma < k.
\end{cases}
$$
\end{lemma}
\begin{proof}
If $z = 0 \, \,$, $p^{\alpha} \| t - z^2$. Suppose $z \neq 0$ with $z = p^{\delta}\epsilon$. Then,  (\ref{lem2}) holds and the statement for $\gamma$ even and less than $\alpha$ follows in the same way as for lemma \ref{lemma2}.
When $\delta = \alpha/2$, $t - z^2 = p^{\alpha}(\beta - \epsilon^2)$. As $\Legendre{\beta}{p} = -1$, $\beta - \epsilon^2 \neq 0 \pmod p$ and so $p^{\alpha} \| t - z^2$. Therefore, $p^{\alpha} \| t - z^2$ when $\alpha/2 \leq \delta < k$. The statement for the case $\gamma = \alpha$ follows from:
\begin{align*}
\lvert \{ z: p^{\delta} \| z  \, \, \text{ for } \frac{\alpha}{2} \leq \delta < k \} \rvert &= \sum_{\delta = \alpha/2}^{k-1}(p-1) p^{k - 1 - \delta} = \, \, p^{k - \alpha/2} - 1.
\end{align*}
If $\gamma$ is odd or $\gamma > \alpha$, it is clear from (\ref{lem2}) that $ord_p( t - z^2) \neq \gamma$.
\end{proof}

\bigskip

\begin{lemma}
\label{lemma4}
If $\alpha$ is even, $\Legendre{\beta}{p} = 1$ and $\gamma:0 \leq \gamma < k$ then
$$
\lvert \Big\{ z: p^{\gamma} \| t - z^2 \pmod {p^k} \Big\} \rvert =
\begin{cases}
    (p - 1)p^{k - 1 - \gamma/2} & \text{if } \, \, \gamma \text{ is even and } 0 \leq \gamma < \alpha \\
    0                                         & \text{if } \, \, \gamma \text{ is odd and } 0 \leq \gamma < \alpha \\
    (p - 2)p^{k - 1 - \alpha/2} & \text{if } \, \, \gamma = \alpha \\
    2(p - 1) p^{k - 1 + \alpha/2 - \gamma}       & \text{ if }\, \,  \alpha < \gamma < k.
\end{cases}
$$
\end{lemma}
\begin{proof}
\bigskip
Suppose $z \neq 0$ and write $z = p^{\delta}\epsilon$ with $p \nmid \epsilon$ and $1 \leq \epsilon < p^{k - \delta}$.

When $\gamma$ is even and $0 \leq \gamma < \alpha$, the calculation is the same as for lemmas \ref{lemma2} and \ref{lemma3} above.

If $\gamma$ is odd and less than $\alpha$, (\ref{lem2}) shows there is no $z$ with $p^{\gamma} \| t - z^2$. 

Now, $p^{\alpha} \| t - z^2$ if and only if $z = 0$ or $p^{\alpha/2 + 1} | z$ or $p^{\alpha/2} \| z$ and $p \nmid \beta - \epsilon^2$. We have:
$$
\lvert \{ z: p^{\alpha/2 + 1} | z \} \rvert = p^{k - 1 - \alpha/2} - 1
$$
and
\begin{align*}
\lvert \{ z: p^{\alpha/2} \| z \, \, \text{and} \, \, p \nmid \beta - \epsilon^2 \} \rvert &= \lvert \{ \epsilon: 1 \leq \epsilon < p^{k - \alpha/2} , \, \, p\nmid \epsilon \, \, \text{and} \, \, p \nmid \beta - \epsilon^2 \} \rvert \\
   &= (p - 1) p^{k - 1 - \alpha/2 } - 2 p^{k - 1 - \alpha/2}
\end{align*}
since $\, \, p | \beta - \epsilon^2$ if and only if $\epsilon = \pm \sqrt{\beta} + np$ for some $n$ with $0 \leq n < p^{k - 1 - \alpha/2}$. Adding the three components gives the result for $\gamma = \alpha$.

If $\gamma > \alpha$, then $p^{\gamma} \| t - z^2$ if and only if $p^{\alpha/2} \| z$ and $p^{\gamma - \alpha} \| \beta - \epsilon^2$. We have:
\begin{align*}
\lvert \{ z: p^{\alpha/2} \| z \, \, \text{and} \, \,p^{\gamma - \alpha} \| \beta - \epsilon^2 \} \rvert &= \lvert \{ \epsilon: 1 \leq \epsilon < p^{k - \alpha/2} , \, \,  p \nmid \epsilon,  \, \, p^{\gamma - \alpha} \| \beta - \epsilon^2 \} \rvert \\
   &= 2 (p - 1) p^{k  -1 +\alpha/2 - \gamma }
\end{align*}
since $p^{\gamma - \alpha} \| \beta - \epsilon^2$ if and only if $\epsilon = \pm \sqrt{\beta} + n p^{\gamma - \alpha}$ where $p \nmid n$ and $0 \leq n < p^{k + \alpha/2 - \gamma}$.  

\end{proof}

\bigskip

\subsection{Proof of theorem \ref{m=3:p=1mod4} and theorem \ref{m=3:p=3mod4}}

\bigskip

Let $t \in \mathbb{Z}_{p^k}$ with $t = p^{\alpha}\beta$. Proving theorems \ref{m=3:p=1mod4} and \ref{m=3:p=3mod4} requires a separate calculation for each of the possibilities for $t$. The four possibilities for $t$ are: $t = 0$, $\alpha$ odd, $\alpha$ even and $\Legendre{\beta}{p} = 1$ and $\alpha$ even and $\Legendre{\beta}{p} = -1$. In addition we need to consider the value of $p \pmod 4$. We provide the calculations for a few of the eight possible cases. The calculations for the other cases follow the same pattern.

\bigskip

\subsubsection{The case $p = 1 \pmod 4$ and $t = 0$}

When $t = 0$, equation (\ref{recurrencek}) becomes:
\begin{equation}
\label{m=2t=0}
N_{3}(0, p^k ) =  \sum_{s \in \mathbb{Z}_{p^k}} N_2(s, p^k ) \times \lvert \Big\{ z: - z^2 = s \Big\} \rvert.
\end{equation}

\bigskip

Writing each $s$ (other than $s = 0$) in the sum in equation  (\ref{m=2t=0}) as $s = p^{\gamma}b$ where $p \nmid b$,  we can use theorem \ref{m=2:p=1mod4} and lemma \ref{lemma1}, to get:
\begin{equation*}
N_{3}(0, p^k ) =  p^{k-1}\left(p(k+1) - k \right)*p^{\lfloor k/2 \rfloor}  + \sum_{\gamma \, \, \text{even}} \, \, \sum_{b: \Legendre{b}{p} = -1} (\gamma + 1)(p-1)p^{k-1} * 2p^{\gamma/2}.
\end{equation*}

For fixed $\gamma$, $\lvert \{ s: p^\gamma \| s \text{ and } \Legendre{b}{p} = -1\} \rvert = \frac{1}{2}(p-1)p^{k-\gamma-1}$. Therefore,
\begin{align*}
N_3(0, p^k ) &= p^{k-1}\left(p(k+1) - k \right) * p^{\lfloor k/2 \rfloor}  + (p-1)^2 p^{k-1}\sum_{\gamma \, \, \text{even}} (\gamma+1)p^{k-\gamma/2-1} \\
           &= p^{k-1}\left(p(k+1) - k \right) * p^{\lfloor k/2 \rfloor}  + (p-1)^2 p^{k-1}\sum_{\gamma = 0}^{\lfloor (k-1)/2 \rfloor} (2\gamma + 1)p^{k - \gamma - 1}
\end{align*}
after reindexing. The two parts of the sum are evaluated as:
$$
(p-1)^2 p^{k-1}\sum_{\gamma = 0}^{\lfloor (k-1)/2 \rfloor} \gamma p^{k - \gamma - 1} = p^{2k - 1} - \lfloor \frac{k+1}{2} \rfloor p^{2k - \lfloor (k+1)/2 \rfloor} + \lfloor \frac{k - 1}{2} \rfloor p^{2k - 1 - \lfloor (k+1)/2 \rfloor}
$$
and
$$
(p-1)^2 p^{k-1}\sum_{\gamma = 0}^{\lfloor (k-1)/2 \rfloor} p^{k - \gamma - 1} = p^{2k} - p^{2k-1} - p^{2k - \lfloor (k+1)/2 \rfloor} + p^{2k - 1 - \lfloor (k+1)/2 \rfloor }.
$$

Manipulating the $floor$ and $ceiling$ functions produces the result for $t = 0$ in theorem~\ref{m=3:p=1mod4}.

\bigskip

\subsubsection{The case $p = 3 \pmod 4$ and $t = 0$}

Now consider the case $p = 3 \pmod 4$. Using theorem \ref{m=2:p=3mod4} and lemma \ref{lemma1}:
\begin{align*}
N_3(0, p^k ) & = p^{2\lfloor k/2 \rfloor}*p^{\lfloor k/2 \rfloor} + \sum_{s \neq 0} N_2(s, p^k )*2p^{\gamma} \\
           & = p^{3\lfloor k/2 \rfloor} + \sum_{s \neq 0} (p+1)p^{k-1}*2p^{\gamma}
\end{align*}
where the sums in the two lines above are over $s: s = p^{2\gamma}b$ with $0 \leq 2\gamma < k$, $p \nmid b$, $\Legendre{b}{p} = -1$ and $1 \leq b < p^{k - 2\gamma}$. For fixed $\gamma$, the number of $b$ satisfying the last three conditions is $\frac{1}{2} p^{k-2\gamma-1} (p - 1)$. We therefore have:

\begin{align*}
N_3(0, p^k ) =& p^{3\lfloor k/2 \rfloor} + \sum_{\gamma = 0}^{\lfloor (k-1)/2 \rfloor} \frac{1}{2} p^{k-2\gamma-1} (p - 1)(p+1)p^{k-1}*2p^{\gamma} \\
           =& p^{3\lfloor k/2 \rfloor} + (p+1)p^{k-1} \sum_{\gamma = 0}^{\lfloor (k-1)/2 \rfloor} (p-1) p^{k-\gamma-1}\\
           =& p^{3\lfloor k/2 \rfloor} + (p+1)p^{k-1} \left( p^k - p^{k - 1 - \lfloor (k-1)/2 \rfloor} \right) \\
           =& p ^{2k} + p^{2k-1} + p^{3\lfloor k/2 \rfloor} - (p+1) p^{2k - 2 - \lfloor (k-1)/2 \rfloor}. 
\end{align*}

\bigskip
This is equivalent to the statement for $t = 0$ in theorem \ref{m=3:p=3mod4}.

\bigskip

\subsubsection{The case $p = 1 \pmod 4$ and $\alpha$ odd}

Theorem \ref{m=2:p=1mod4} shows that for $s \in \mathbb{Z}_{p^k}$, $N_2(s, p^k )$ depends only on $ord_p(s)$. We can use (\ref{recurrencek}) and lemmas \ref{lemma1} and  \ref{lemma2} in this case to get:
\begin{align*}
\begin{split}
N_{3}(t, p^k ) =  \, \, &p^{k-1} \left( p(k+1) - k \right) \times \lvert \Big\{ z: t - z^2 = 0 \Big\} \rvert \\
            & \, \, + \, \, \sum_{\gamma = 0}^{k-1} (\gamma + 1)(p - 1)p^{k-1}  \times  \lvert \Big\{ z: p^{\gamma} \| t - z^2 \Big\} \rvert \\
             = &\sum_{\gamma = 0}^{(\alpha-1)/2} (2\gamma + 1)(p - 1)p^{k-1}  \times (p-1) p^{k-1-\gamma} + (\alpha + 1) (p-1) p^{k-1} \times p^{k- (\alpha+1)/2}
\end{split}
\end{align*}

where the sum has been reindexed to account for even $\gamma$. The two parts of the sum are evaluated as:
$$
2 (p-1)^2 p^{2k-2}\sum_{\gamma = 0}^{(\alpha - 1)/2} \gamma p^{- \gamma} = 2 p^{2k - 1} - \left( \alpha  + 1 \right) p^{2k - (\alpha + 1)/2} + \left( \alpha - 1 \right) p^{2k - (\alpha + 3)/2 }
$$
and
$$
(p-1)^2 p^{2k-2}\sum_{\gamma = 0}^{(\alpha - 1)/2} p^{- \gamma} = (p - 1) (p^{2k - 1} - p^{2k - (\alpha + 3)/2}).
$$
The formula in theorem \ref{m=3:p=1mod4} for the case when $\alpha$ is odd follows.

\bigskip

\subsubsection{The case $p = 1 \pmod 4, \, \, \alpha$ even and $\Legendre{\beta}{p} = 1$}

We use (\ref{recurrencek}), theorem \ref{m=2:p=1mod4}, lemmas \ref{lemma1} and  \ref{lemma4} in this case to get:

\begin{align*}
N_{3}(t, p^k ) &= \Sigma_1 + \Sigma_2 + \Sigma_3 + \Sigma_4
\end{align*}

where, $\Sigma_1$ is the contribution to $N_3(t, p^k )$ in (\ref{recurrencek}) from $s = 0$, $\Sigma_2$ is the contribution from $s$ such that $0 \leq ord_p(s) < \alpha$, $\Sigma_3$ comes from $ord_p(s) = \alpha$ and $\Sigma_4$ comes from $s$ with $\alpha < ord_p(s) < k$.
So,
\begin{align*}
\Sigma_1 = p^{k-1} \Big( p(k+1) - k \Big) \times 2 p^{\alpha/2}.
\end{align*}

After re-indexing for even $\gamma$ in $\Sigma_2$,
\begin{align*}
\Sigma_2 &= (p - 1) p^{k - 1} \sum_{\gamma = 0}^{\alpha/2 - 1} (2 \gamma + 1) (p - 1) p^{k - 1 - \gamma} \\
                & = p^{2k} + p^{2k - 1} - (\alpha+1) p^{2k-\alpha/2} + (\alpha - 1) p^{2k - 1 - \alpha/2}.
\end{align*}
We also have:
\begin{align*}
\Sigma_3 &= (p - 1) (p - 2) (\alpha + 1)  p^{2k - 2 - \alpha/2}
\end{align*}
and \begin{align*}
\Sigma_4 &= 2 (p - 1)^2 p^{k - 1} \sum_{\gamma = \alpha + 1}^{k - 1} (\gamma + 1) p^{k - 1 + \alpha/2 - \gamma} \\
                & = 2(\alpha + 2) p^{2k - 1 - \alpha/2} - 2(k + 1) p^{k + \alpha/2} + 2k p^{k - 1 + \alpha/2} - 2(\alpha + 1) p^{2k - 2 - \alpha/2}.
\end{align*}

Adding up the various terms gives the formula for $N_3(t, p^k )$ in theorem \ref{m=3:p=1mod4} when $\alpha$ is even and $\Legendre{\beta}{p} = 1$. 

\bigskip

\bibliographystyle{plain}
\begin{small}
\bibliography{SSFunction}
\end{small}

\end{document}